\documentclass[12pt]{amsart}
\usepackage{graphicx}

\usepackage{amsmath, mathrsfs, amssymb, eucal, amsthm, amscd, amsxtra, amsthm, geometry, enumerate,mathtools}
\usepackage{amsrefs}
\usepackage{colonequals}
\usepackage{hyperref}

\usepackage{aliascnt}

\hypersetup{
colorlinks=true,
linkcolor=blue,
citecolor=red,
filecolor=magenta,      
urlcolor=cyan,
pdftitle={Mordell_Curves_Arithmetic_Progression},
pdfpagemode=FullScreen,
}

\usepackage{amsrefs}

\usepackage{cleveref}

\makeatletter

\numberwithin{equation}{section}
\newtheorem{theorem}{\bf Theorem}[section]
\newtheorem{lemma}[theorem]{\bf Lemma}

\newtheorem{claim}[theorem]{\bf Claim}
\theoremstyle{remark}
\newtheorem{rem}[theorem]{\bf Remark}
\newtheorem{exmp}[theorem]{\bf Example}

\DeclareMathAlphabet{\mathcal}{OMS}{cmsy}{m}{n}

\def \Q {{\mathbb{Q}}}
\def \Z {{\mathbb Z}}

\title{Mordell Curves with Ordinates in Arithmetic Progression}
\author[S. A. Alkabouss]{Sakha A. Alkabouss}
\address{Universit\'e Andr\'e Salifou de Zinder, Niger}
\email{sakha.aboussaghidalkabouss@uas.edu.ne}
\author[B. Earp-Lynch]{Benjamin Earp-Lynch}
\author[S. Earp-Lynch]{Simon Earp-Lynch}
\address{Carleton University, Department of Mathematics, Canada}
\email{benjaminearplynch@cmail.carleton.ca}
\email{simonearplynch@cmail.carleton.ca}
\author[O. Kihel]{Omar Kihel}
\address{Brock University, Department of Mathematics and Statistics, Canada}
\email{okihel@brocku.ca}
\begin{document}

\begin{abstract}
We show that Mordell curves with arithmetic progressions in the $y$-coordinate of length $7$ have not been ruled out by previous work, and we give non-isomorphic families of Mordell curves with $y$-arithmetic progressions of length $6$. We also construct parametric families of elliptic curves of moderate rank, with subfamilies possessing rational points in arithmetic progression.
\end{abstract}
\maketitle
\section{Introduction}

Elliptic curves may have rational or integral points lying in arithmetic progression. One may consider an arithmetic progression in either of the coordinates, or in both simultaneously. We will refer to an arithmetic progression in the $x$ or $y$-coordinate of an elliptic curve as an $x$ or $y$-arithmetic progression. The length of an $x$ or $y$-arithmetic progression has some bearing on a curve's algebraic rank. Indeed, it was shown by Garcia-Fritz and Pasten \cite{MR4259152} that there is a lower bound on the rank of an elliptic curve which depends on the maximal length of an arithmetic progression in either of its coordinates.

A Mordell elliptic curve is a curve with Weierstrass form $$E_{k}:\quad y^{2}=x^{3}+k,\quad k\in\Q.$$ Mohanty \cite{MR0612198} conjectured that for $k\in \Z$ the maximum length of an $x$ or $y$-arithmetic progression of $E_{k}$ is $4$. Mohanty's conjecture was later amended by Lee and Velez \cite{MR1200840}, who proved the existence of infinitely many curves with $y$-arithmetic progressions of length $6$ and found examples of such curves. Dey and Maji \cite{MR3555877} investigated arithmetic progressions on Mordell curves $E_{k}$ for $k\in \Q$. In particular, they claimed that $6$ is the maximum possible length of a $y$-arithmetic progression, and constructed a family of curves with such a property.

In this work, we comment on both of the aforementioned claims of Dey and Maji. In \Cref{Claim1}, we show that $y$-arithmetic progressions of length at least $7$ on Mordell curves have not been ruled out. In \Cref{Claim2}, we give a method for constructing nonisomorphic families of Mordell curves containing a $y$-arithmetic progression of length $6$. In \Cref{MordellSec}, we exhibit a parametric family of Mordell curves of rank $4$, some parametric subfamilies of which have ordinates in arithmetic progression. In \Cref{SelmerSec} we discuss a parametric family of Selmer curves of rank $3$.  For this family as well, we may choose parameters so that the resulting curves possess points in arithmetic progression.

\section{$y$-Arithmetic Progressions on Mordell Curves}

We consider the Mordell elliptic curve $$E_{k}:\quad y^{2}=x^{3}+k,\quad k\in\Q.$$  Define $S_{y}(E_{k})$ to be the maximal length of a $y$-arithmetic progression on $E_{k}$.  We address an earlier claim of Dey and Maji \cite[Theorem 2]{MR3555877}, which we state here.
\begin{claim}\cite[Theorem 2]{MR3555877}\label{PreviousThm}
Let $E:y^{2}=x^{3}+k$ be an elliptic curve for some $k\in\Q$.  Then $S_{y}(E)\leq6$.  Moreover, there exist infinitely many such elliptic curves with $S_{y}(E)=6$.
\end{claim}

We comment on both parts of the statement of \Cref{PreviousThm}.
\subsection{First part of \Cref{PreviousThm}}\label{Claim1}
\Cref{PreviousThm} states that there is no Mordell curve possessing seven points in $y$-arithmetic progression.  However, the proof is incorrect.  The authors point out that for there to be $7$ points in $y$-arithmetic progression on a Mordell curve $E_{k}$, there must be a rational solution to one or more equations of the form
\begin{equation}\label{DiagonalCubic}
x^{3}+y^{3}+cz^{3}=c,\quad c\in\Q,\quad c\neq0,1,2.
\end{equation}  They then employ the following theorem of Bremner \cite[Theorem 1]{MR0958252}, to simultaneous equations of the form \eqref{DiagonalCubic}.  We have modified the statement slightly to suit our requirements.  In its original form, the theorem lists two additional families of solutions in the case $c=2$.
\begin{lemma}\cite[Theorem 1]{MR0958252}\label{BremnerThm}
The only rational solution curves for the surface
\begin{equation}\label{BremnerEq}
V:x^{3}+y^{3}+cz^{3}=c\quad (c\neq1,2)
\end{equation} are given up to symmetry by
\begin{align*}
(i)\quad (x,y,z)=\left(\lambda, -\lambda, 1\right) \quad \text{and}\quad (ii)\quad (x,y,z)=\left(3\lambda-\frac{9}{c}\lambda^{4},\frac{9}{c}\lambda^{4},1-\frac{9}{c}\lambda^{3}\right),
\end{align*}
for $\lambda\in\Q$.
\end{lemma}
\noindent We note that in earlier works \cites{MR0958252,MR3555877}, the family $(ii)$ is written incorrectly.

In the proof of \Cref{PreviousThm}, it was assumed that \Cref{BremnerThm} gives all solutions to the Equations \eqref{BremnerEq}.  This interpretation of \Cref{BremnerThm} is incorrect.  Quite aside from not addressing the possible existence of sporadic solutions that do not belong to any parametric family, there are in fact infinitely many parametric families of solutions.  It was noted in the abstract of Bremner's article \cite{MR0958252} that, while the surfaces $x^{3}+y^{3}+cz^{3}=c$ only contain the $2$ (or $4$) listed polynomial parametric solutions corresponding to curves of arithmetic genus $0$ on the surface, they contain infinitely many polynomial parametric solutions corresponding to curves of arithmetic genus greater than $0$.  This is later proven explicitly in an unnumbered theorem on the final pages of the same article \cite{MR0958252}.  In the proof of that theorem, the author gives an example of a recursively-defined infinite family of parametric solutions.  Since the proof of \Cref{PreviousThm} \cite[Theorem 2]{MR3555877} relies on this erroneous assumption, we conclude that the existence of a Mordell curve possessing seven points with $y$-coordinates in arithmetic progression has not been ruled out.

\subsection{Second part of \Cref{PreviousThm}}\label{Claim2}
The second statement made in \Cref{PreviousThm} is that there are infinitely many Mordell elliptic curves with $6$ points in $y$-arithmetic progression.  This is true, however the previous method of proof \cite[Theorem 2]{MR3555877} only pointed to one distinct isomorphism class of such curves.
\begin{rem}
$E_{k}$ and $E_{k'}$ are isomorphic over $\Q$ if and only if $\frac{k'}{k}\in\Q^{6}$.
\end{rem}
We outline here a method to obtain infinitely-many non-isomorphic Mordell curves possessing a $y$-arithmetic progression of length $6$.

Consider the cubic $$C:\quad 3U^{3}-V^{3}=2.$$  It contains the rational point $(1,1)$, so it is an elliptic curve.

\begin{theorem}\label{infmanymordell}
Let $(U,V)\in C(\Q)$, with $U\neq1$.  Let $N=\frac{U^{3}-1}{8}\in\Q$, and for any $q\in\Q$ define $$d=Nq^{2},\quad r=N^{2}q^{3},\quad k=r^{2}-d^{3}=N^{3}(N-1)q^{6}.$$  Then the elliptic curve $$E_{k}:\quad y^{2}=x^{3}+k$$ contains the six rational points $$(d,\pm r),\quad (Ud,\pm3r),\quad (Vd,\pm5r)$$ whose $y$-coordinates form an arithmetic progression of length six.
\end{theorem}

\begin{proof}
The points $(d,\pm r)$ are on $E_{k}$ by definition.  Substituting $x=Ud$ and $k=r^{2}-d^{3}$ in the right side of $E_{k}$ gives
$$
U^{3}d^{3}+r^{2}-d^{3}=8Nd^{3}+r^{2}=9r^{2}=(\pm3r)^{2}.
$$
Similarly, substituting $x=Vd$ yields $$V^{3}d^{3}+r^{2}-d^{3}=24Nd^{3}+r^{2}=25r^{2}=(\pm5r)^{2}.$$  So we have verified that $E_{k}$ contains the six listed rational points.  Their $y$-coordinates form the arithmetic progression $-5r,-3r,-r,r,3r,5r$ of length $6$.
\end{proof}
The proof of \Cref{PreviousThm} \cite[Theorem 2]{MR3555877} uses a specific point $(U,V)=\left(\frac{1}{4},\frac{-5}{4}\right)\in C$, and constructs a family of curves with six points in $y$-arithmetic progression.  Our \Cref{infmanymordell} is more general in that it shows that any rational point on $C$ produces such a family.

Let $P=(U,V)\in C(\Q)$, and define $A(P)=N^{3}(N-1)\in\Q$, and $$E_{P,q}:\quad y^{2}=x^{3}+A(P)q^{6}.$$  The curves corresponding to $P$ and $P'$ are isomorphic if and only if $$\frac{A(P')}{A(P)}\in\Q^{6}.$$

\begin{exmp}
The proof of \Cref{PreviousThm} \cite[Theorem 2]{MR3555877} fixed the point $P_{0}=\left(\frac{1}{4},\frac{-5}{4}\right)=(U,V)$.  This leads to $$N_{0}=\frac{U^{3}-1}{8}=\frac{\left(1/4\right)^{3}-1}{8}=\frac{-63}{512},$$ and $$A(P_{0})=N_{0}^{3}(N_{0}-1)=\frac{143777025}{68719476736}.$$  So $$E_{P_{0},q}:\quad y^{2}=x^{3}+\frac{143777025}{68719476736}q^{6}$$ contains a $y$-progression of length $6$.  Taking $q'=\frac{32}{3}q$ then gives the same curves as were found in the proof of \Cref{PreviousThm} \cite[Theorem 2]{MR3555877}.
\end{exmp}

\begin{exmp}
We produce another rational point $P_{1}=\left(-\frac{253}{488},-\frac{655}{488}\right)\in C(\Q)$, which leads to $$N_{1}=\frac{U_{1}^{3}-1}{8}=-\frac{132408549}{929714176}.$$  Then $$A(P_{1})=N_{1}^{3}(N_{1}-1)=\frac{2465600901357492654209519114711025}{747132815722366705449761064223768576}$$ and we obtain $$E_{P_{1},q}:\quad y^{2}=x^{3}+A(P_{1})q^{6}.$$
\end{exmp}

\begin{exmp}
Using the tangent to the point $P_{1}$, we construct another rational point $$P_{2}=\left(\frac{129900299507}{160841972528},-\frac{120418942015}{160841972528}\right)\in C(\Q).$$  We can then compute $$N_{2}=\frac{U_{2}^{3}-1}{8}=-\frac{1969055306866558092207723045065109}{33288034940430841204248215754735616}$$ and $A(P_{2})$.  We do not reproduce the latter since its height is too large for it to fit within the margins.
\end{exmp}

We constructed $3$ families of elliptic curves and we can verify that these families are not isomorphic over $\Q$ by checking that $$\frac{k'}{k}=\frac{A(P')}{A(P)}\not\in\Q^{6}.$$

We now demonstrate a pair of elementary constructions of Mordell elliptic curves of moderate rank.

\section{A Rank 4 Family of Mordell Curves}\label{MordellSec}

Start with a curve $$E:\quad y^2=lx^3+k$$ and consider four distinct rational points $(z,q),(w,r),(v,s)$ and $(u,t)$ on $E$. We write:
\begin{align*}
r^2=lw^3+k,\quad s^2=lv^3+k,\quad t^2=lu^3+k \quad \text{and} \quad q^2=lz^{3}+k.
\end{align*}
Solving for $l$ and $k$ gives $$l=\frac{t^2-s^2}{u^3-v^3}\quad \text{and} \quad k=\frac{s^2u^3-t^2v^3}{u^3-v^3}.$$  Hence
$$
r^{2}=\frac{t^2-s^2}{u^3-v^3}w^{3}+\frac{s^2u^3-t^2v^3}{u^3-v^3}\quad \text{and}\quad q^{2}=\frac{t^2-s^2}{u^3-v^3}z^{3}+\frac{s^2u^3-t^2v^3}{u^3-v^3},
$$
and rearranging these equations yields
\begin{align*}
&(r^{2}-s^{2})u^{3}+(s^{2}-t^{2})w^{3}+(t^{2}-r^{2})v^{3}=0 \quad \text{and}\\
&(q^{2}-s^{2})u^{3}+(s^{2}-t^{2})z^{3}+(t^{2}-q^{2})u^{3}=0.
\end{align*}
From this, we get two equations
\begin{align}
&(r^{2}-q^{2})u^{3}+(s^{2}-t^{2})(w^{3}-z^{3})+(q^{2}-r^{2})v^{3}=0,\quad\text{and}\label{PreBrem1}\\
&(r^{2}+q^{2}-2s)u^{3}+(s^{2}-t^{2})(w^{3}+z^{3})+(2t^{2}-r^{2}-q^{2})v^{3}=0.\label{PreBrem2}
\end{align}
Provided $v\neq 0$, equation \eqref{PreBrem1} can be rearranged to obtain \begin{equation}\label{BremMordellEqn}
\left(\frac{w}{v}\right)^{3}+\left(\frac{-z}{v}\right)^{3}+\left(\frac{r^{2}-q^{2}}{s^{2}-t^{2}}\right)\left(\frac{u}{v}\right)^{3}=\left(\frac{r^{2}-q^{2}}{s^{2}-t^{2}}\right).
\end{equation}
Write $c\colonequals\frac{r^{2}-q^{2}}{s^{2}-t^{2}}$. If $r^{2}-q^{2}\neq 0$ and $\lambda \in \Q$, then $$\left(\frac{w}{v}\right)=3\lambda-\frac{9}{c}\lambda^{4}, \quad \left(\frac{z}{v}\right)=\frac{-9}{c}\lambda^{4},\quad \text{and} \quad \left(\frac{u}{v}\right)=1-\frac{9}{c}\lambda^{3}$$ is a parameterized family of solutions to equation \eqref{BremMordellEqn} (see \Cref{BremnerThm}).  Substituting this into equation \eqref{PreBrem2} gives: $$\left(r^{2}+q^{2}-2s^{2}\right)\left(1-\frac{9}{c}\lambda^{3}\right)^{3}+\left(s^{2}-t^{2}\right)\left(\left(3\lambda-\frac{9}{c}\lambda^{4}\right)^{3}-\left(\frac{9}{c}\lambda^{4}\right)^{3}\right)=r^{2}+q^{2}-2t^{2},$$ whence \begin{equation}\label{preAPeqn}\left(r^{2}+q^{2}-2s^{2}\right)\left(\left(1-\frac{9}{c}\lambda^{3}\right)^{3}-1\right)+\left(s^{2}-t^{2}\right)\left(\left(3\lambda-\frac{9}{c}\lambda^{4}\right)^{3}-\left(\frac{9}{c}\lambda^{4}\right)^{3}-2\right)=0.\end{equation}

Now suppose that the following relationship holds between $q,r,s$ and $t$:  $$q=t+\alpha,\quad r=t+\beta,\quad s=t+(\alpha+\beta),$$ for rational numbers $\alpha,\beta$.  Then we can eliminate $t$ from $c$, obtaining $$c=\frac{r^2-q^2}{s^2-t^2}=\frac{\beta-\alpha}{\alpha+\beta},$$ as well as
\begin{align*}
r^2+q^2-2t^2&=2(\alpha+\beta)t+\alpha^{2}+\beta^{2},\\
r^2+q^2-2s^2&=-2(\alpha+\beta)t-\alpha^{2}-\beta^{2}-4\alpha\beta,\text{ and}\\
s^2-t^2&=2(\alpha+\beta)t+(\alpha+\beta)^{2}.
\end{align*}
Using \eqref{preAPeqn} to solve for $t$ gives $t=\frac{(\alpha+\beta)g-(\beta+2\alpha)h}{2(h-g)}$, where
\begin{align*}
g&=729(\alpha+\beta)^3\lambda^{12}+(\beta-\alpha)^3\text{ and}\\
h&=27\beta(27( \alpha + \beta )^2\lambda^9 - 9(\beta^2-\alpha^2)\lambda^6 + (\beta-\alpha)^2\lambda^3).
\end{align*}
It follows that the points $\left(z,t+\alpha\right),\left(w,t+\beta\right),\left(v,t+\alpha+\beta\right)$ and $\left(u,t\right)$ all lie on the elliptic curve $$E_{\alpha,\beta,\lambda}:\quad y^{2}=lx^3+k,$$ for $$l=\frac{2\beta t + \beta ( \alpha + \beta )}{h},\text{ and }k=\frac{ht^2 + ( 2t + \alpha + \beta ) ( h (\alpha+\beta)-\beta(\beta-\alpha)^3))}{h},$$ where
\begin{align*}
u &= (\beta-\alpha)-9(\alpha+\beta)\lambda^{3},& v&=\beta-\alpha,\\
w&=3(\beta-\alpha)\lambda-9(\alpha+\beta)\lambda^{4},&z&=-9(\alpha+\beta)\lambda^{4},
\end{align*} for rational parameters $\alpha,\beta,\lambda$.

Performing the change of variables $$y=\frac{Y}{l},\quad x=\frac{X}{l},$$ and substituting $(\alpha,\beta,\lambda)=(2,-1,1)$, we obtain the curve $$E_{2,-1,1}:\quad Y^{2}=X^{3}+73/5082121521$$ with the four points
\begin{align*}
(u,t)&=\left(\frac{8}{801}, \frac{215}{213867}\right),&(v,s)&=\left(\frac{2}{801}, \frac{37}{213867}\right),\\
(w,r)&=\left(\frac{4}{267}, \frac{131}{71289}\right),&(z,q)&=\left(\frac{2}{267}, \frac{-47}{71289}\right).
\end{align*}
We find that the regulator of these four points is approximately $37.23$, so they are independent, and the rank of $E_{2,-1,1}$ is at least $4$ (indeed, we verify with SageMath \cite{sagemath} that the rank is exactly $4$ in this case).  By the specialization theorems of N\'eron and Silverman \cite[IV\S11]{MR1312368}, the generic rank of $E_{\alpha,\beta,\lambda}$ is at least $4$.  For certain choices of parameters, for example $\alpha=2d,\beta=-d$, this is an example of a parametric family of elliptic curves with four points whose ordinates are in arithmetic progression.

\section{A Rank 3 Family of Selmer Curves}\label{SelmerSec}
Start with a Selmer curve
\begin{equation}\label{SelmerStart}
ax^3+by^3=cz^3,
\end{equation} where $a,b,c\in\Q$.  Dividing \eqref{SelmerStart} by $cy^3$, we get $$\frac{a}{c}\left(\frac{x}{y}\right)^{3}+\frac{b}{c}=\left(\frac{z}{y}\right)^{3}.$$  For the choice of variables $X=\frac{x}{y}$ and $Y=\frac{z}{y}$, we then have the curve
\begin{equation}\label{SelmerSecond}
Y^{3}=AX^{3}+B,
\end{equation}
where $A=\frac{a}{c}$ and $B=\frac{b}{c}$.  Suppose that this curve has the three rational points $(X,Y)=(r,w),(s,v)$ and $(t,u)$.  We obtain three equations
\begin{align}
w^{3}&=Ar^{3}+B,\label{rweq}\\
v^{3}&=As^{3}+B,\label{sveq}\\
u^{3}&=At^{3}+B.\label{tueq}
\end{align}
Equation \eqref{tueq} gives $B=u^{3}-At^{3}$, which we substitute into \eqref{sveq} to get $$A=\frac{v^{3}-u^{3}}{s^{3}-t^{3}}.$$  Substituting this expression into the expression for $B$ gives $$B=\frac{s^{3}u^{3}-t^{3}v^{3}}{s^{3}-t^{3}},$$ and so $$w^{3}=\left(\frac{v^{3}-u^{3}}{s^{3}-t^{3}}\right)r^{3}+\frac{s^{3}u^{3}-t^{3}v^{3}}{s^{3}-t^{3}},$$  whence $$(s^3-t^3)w^{3}=(r^3-t^3)v^{3}+(s^3-r^3)u^{3}.$$  If $v\neq0$, this implies that $(\frac{u}{v},\frac{w}{v})$ is a point on the curve $$C:d_{1}Y_{1}^{3}+d_{2}X_{1}^{3}=d_{1}+d_{2},$$ for $d_{1}=s^{3}-t^3$ and $d_{2}=r^3-s^3$.  The tangent line to this curve at the point $(1,1)$ has equation $$d_{1}Y_{1}+d_{2}X_{1}=d_{1}+d_{2},$$ and also meets the curve $C$ at the point $$(X_{1},Y_{1})=\left(-\frac{2d_1 + d_2}{d_1 - d_2},\frac{d_{1}+2d_{2}}{d_{1}-d_{2}}\right),$$ and so we can take
\begin{align*}
u = 2t^3-r^3-s^3,\\
v = 2s^3-r^3-t^3,\\
w = 2r^3-s^3-t^3.
\end{align*}
Substituting these values into the expressions for $A$ and $B$, we get
\begin{align*}
A&=9( r^6 - r^3s^3 + s^6 - r^3t^3 - s^3t^3 + t^6 ),\\
B&=-(r+s+t)(r^2 + 2rs + s^2 - rt - st + t^2)(r^2 - rs + s^2 + 2rt - st + t^2)\\&\quad\hspace*{4.75mm}(r^2 - rs + s^2 - rt + 2st + t^2)(r^2 - rs + s^2 - rt - st + t^2).
\end{align*}
Specializing to the values $r=1,s=2,t=3$ in \eqref{SelmerSecond} gives the curve $$\left(\frac{Y}{18}\right)^3=181\left(\frac{X}{6}\right)^3-7.$$  Multiplying by $181^2\left(\frac{Y}{18}\right)^{3}$ gives $$\left(181\left(\frac{Y}{18}\right)^{3}\right)^{2}+7\cdot181\left(181\left(\frac{Y}{18}\right)^{3}\right)=\left(\frac{181XY}{108}\right)^{3}.$$  Setting $X_{0}=\frac{181XY}{108}$ and $Y_{0}=181\left(\frac{Y}{18}\right)^{3}$ then gives the elliptic curve in Weierstrass form: $$E_{1,2,3}:\quad Y_{0}^{2}+7\cdot181Y_{0}=X_{0}^{3},$$ which has the points $$(X_{0},Y_{0})=\left(\frac{-1991}{36},\frac{-240911}{216}\right),\left(\frac{-362}{9}, \frac{-1448}{27}\right),\left(\frac{905}{4}, \frac{22625}{8}\right).$$  These points have regulator equal to approximately $45.0471$ and thus are independent.

In general, the curves
\begin{equation}\label{WeierstrassForm}
E_{r,s,t}:\quad y^{2}-ABy=x^3
\end{equation}
are related by a 3-isogeny to the curves in \eqref{SelmerSecond} and, by the specialization theorems of N\'eron and Silverman \cite[IV\S11]{MR1312368}, have three independent points over $\Q(r,s,t)$.  Since isogenies preserve rank, it follows that the family of Selmer curves \eqref{SelmerSecond} will have generic rank at least 3.

\subsection{Specialization and Rank Bounds for $E_{r,s,t}$ }
We include some remarks which point to why these curves might be of interest. 
Specializing to, for example, $(r,s,t)=(u-1,u+1,u)$, we obtain the family
$$E_{u}:\quad Y^2+a(u)Y=X^3$$
over $\Q(u)$, where
$$a(u)=9u^{3}(9u^{4}+9u^{2}+1)(u^{2}+u+1)(u^{2}-u+1).$$ 
The $j$-invariant of $E_{u}$ is $j=0$ and so the family of elliptic curves is isotrivial: the fibres are all isomorphic over $\overline{\Q}$ to a fixed elliptic curve $$y^{2}=x^{3}+b,\quad b\text{ fixed},$$ their endomorphism group contains $\Z[\zeta_{3}]$ and so they have complex multiplication by $\Z[\zeta_{3}]$.

Using the change of variables $Y=y'-\frac{a(u)}{2}$, we obtain
$$y'^{2}=x^3+b(u)\quad\text{for}\quad b(u)=\left(\frac{a(u)}{2}\right)^{2}.$$
Writing $$x=c^{2}x'\quad\text{and}\quad y'=c^{3}y''$$ for $$c=\left(b(u)\right)^{1/6},$$ we obtain a curve isomorphic to $$y''^{2}=x'^{3}+1$$ 
over $\overline{\Q(u)}$, and all the fibres are isomorphic over $\overline{\Q(u)}$.

\begin{rem} It is known that isotrivial families may have a very large rank. \end{rem}

The standard Weierstrass quantities for $E_u$ are $c_4 = 0$ and $c_6 = -54a(u)^2$, giving the discriminant of $E_{u}$ as 
$$\Delta=-\frac{27}{16}(a(u))^{4},$$ 
so $E_{u}$ may have bad reduction when $a(u)=0$. 
The formula of Shioda-Tate implies that
$$\operatorname{rank}\left(E\left(\Q(u)\right)\right)=\rho\left(\operatorname{NS}E_{u}\right)-2-\sum\limits_{v}\left(m_{v}-1\right),$$
where
\begin{itemize}
\item{} $\rho\left(\operatorname{NS}E_{u}\right)$ is the rank of the N\'{e}ron-Severi group;
\item{} the $v$ are given by the places where $E_{u}$ has bad reduction;
\item{} $m_{v}$ is the number of components of the fibre of $E_{u}$ at $v$.
\end{itemize}

At $u=0$, $v_0(a(u))=3$, so $v_0(\Delta)=12$ and $v_0(c_6)=6$. The Weierstrass model is not minimal at this place. Applying the change of variables $X \to u^2 X$ and $Y \to u^3 Y$ gives a new minimal discriminant with valuation $0$, meaning $E_u$ has good reduction at $u=0$. Thus, the distinct places of bad reduction are
\begin{itemize}
\item{} the $4$ roots of $9u^{4}+9u^{2}+1=0$,
\item{} the $2$ roots of $u^{2}+u+1=0$,
\item{} the $2$ roots of $u^{2}-u+1=0$,
\item{} and $u=\infty$.
\end{itemize}

There are $9$ distinct singular fibres. Since $v(\Delta)=4$ and $v(c_6)=2$ at each of these finite places (and similarly at $u=\infty$), each singular fibre is of additive type. 
Using the classification of Kodaira-N\'{e}ron, the singular fibres are of type IV. 
It is known that when the fibre type is IV, there are $3$ components. 
So $m_{v}-1=3-1=2$, and the N\'{e}ron-Severi formula gives 
$$\operatorname{rank} E\left(\Q(u)\right)=\rho-2-\sum\limits_{v}\left(m_{v}-1\right)=\rho-2-18=\rho-20.$$

Consequently, the generic rank of this isotrivial family can be quite large, making specialized subfamilies candidates for yielding rational points in arithmetic progression longer than those of length $3$ examined in the first part of the section.

\end{document}